\nonstopmode

\documentclass[12pt]{article}

\usepackage{amsmath,amssymb,amsthm}
\usepackage{graphics,epsfig,calc}
\usepackage{amsmath}
\usepackage{latexsym,epsfig,bm,amssymb}
\usepackage{color}
\usepackage{amsthm,mathrsfs}

\usepackage{mathptmx}

\DeclareSymbolFont{AMSb}{U}{msb}{m}{n}
\DeclareSymbolFontAlphabet{\mathbb}{AMSb}

\newcommand{\cK}{{\cal K}}
\newcommand{\cM}{{\cal M}}
\newcommand{\cS}{{\cal S}}
\newcommand{\cV}{{\cal V}}
\newcommand{\cW}{{\cal W}}
\newcommand{\cX}{{\cal X}}
\newcommand{\al}{\alpha}

\newcommand{\ci}{\cite}
\newcommand{\de}{\delta}
\newcommand{\De}{\Delta}
\newcommand{\ds}{\displaystyle}
\newcommand{\fr}{\frac}

\newcommand{\la}{\label}
\newcommand{\lam}{\lambda}

\newcommand{\na}{\nabla}

\newcommand{\vp}{\varphi}

\newcommand{\ov}{\overline}
\newcommand{\pa}{\partial}
\newcommand{\re}{\ref}

\newcommand{\Si}{\Sigma}
\newcommand{\si}{\sigma}

\newcommand{\ve}{\varepsilon}

\newcommand\C{{\mathbb C}}
\newcommand\R{{\mathbb R}}
\newcommand\N{{\mathbb N}}
\newcommand\Z{{\mathbb Z}}

\newcommand{\vka}{\varkappa}
\newcommand{\cm}{{\rm m}}

\newcommand{\5}{{\hspace{0.5mm}}}

\newcommand{\rRe}{{\rm Re\5}}
\newcommand{\rIm}{{\rm Im\5}}
 \newcommand{\st}{\stackrel}

\newcommand{\toLt}{\st{L^2({T_N})}{-\!\!\!-\!\!\!-\!\!\!\!\longrightarrow}} 
\newcommand{\toLwt}{\st{L^2_w({T_N})}{-\!\!\!-\!\!\!-\!\!\!\!\longrightarrow}} 
 
\newcommand{\toLC}{\st{C({T_N})}{-\!\!-\!\!\!\!\longrightarrow}} 
\newcommand{\tocX}{\st{\cX}{-\!\!\!\!\longrightarrow}}

\newcommand{\Ker}{{\rm Ker\5}}

\renewcommand{\theequation}{\thesection.\arabic{equation}}
\newtheorem{theorem}{Theorem}[section]
\renewcommand{\thetheorem}{\arabic{section}.\arabic{theorem}}
\newtheorem{definition}[theorem]{Definition}

\newtheorem{lemma}[theorem]{Lemma}

\newtheorem{remark}[theorem]{Remark}
\newtheorem{remarks}[theorem]{Remarks}
\newtheorem{cor}[theorem]{Corollary}
\newtheorem{proposition}[theorem]{Proposition}

\newcommand{\bd}{\begin{definition}}
 \newcommand{\ed}{\end{definition}}
\newcommand{\bt}{\begin{theorem}}
 \newcommand{\et}{\end{theorem}}
\newcommand{\bp}{\begin{proposition}}
 \newcommand{\ep}{\end{proposition}}
\newcommand{\bl}{\begin{lemma}}
 \newcommand{\el}{\end{lemma}}
\newcommand{\bc}{\begin{cor}}
 \newcommand{\ec}{\end{cor}}
\newcommand{\br}{\begin{remark} }
 \newcommand{\er}{\end{remark}}
\newcommand{\brs}{\begin{remarks} }
 \newcommand{\ers}{\end{remarks}}

\begin{document}

\begin{titlepage}
\vspace{2cm}

\begin{center}
{\Large\bf 
On   stability of ground states for  finite crystals
\medskip\\
 in the Schr\"odinger-Poisson model
}
\end{center}
\bigskip\bigskip

 \begin{center}
{\large A. Komech}
\footnote{
Supported partly 
by 
Austrian Science Fund (FWF): P28152-N35,
and the grant of  RFBR 16-01-00100.}
\\
{\it Faculty of Mathematics of Vienna University\\
and Institute for Information Transmission Problems RAS } \\
e-mail:~alexander.komech@univie.ac.at
\bigskip\\
{\large E. Kopylova}
\footnote{
Supported partly by  
Austrian Science Fund (FWF): P27492-N25,
and the grant of RFBR 16-01-00100.}
\\
{\it Faculty of Mathematics of Vienna University\\
and Institute for Information Transmission Problems RAS} \\
 e-mail:~elena.kopylova@univie.ac.at
\end{center}
\vspace{1cm}

 \begin{abstract}
We consider the 
Schr\"odinger-Poisson-Newton equations for finite  crystals
under periodic boundary conditions with one ion per cell of a lattice.
The electrons are described by one-particle Schr\"odinger equation.

Our main results are 
i) the global dynamics with  moving ions;
ii) the  orbital stability of periodic  ground state under a novel  
 Jellium and Wiener-type conditions on the ion charge density.
 Under the Jellium condition both ionic and  electronic charge densities  
 for the ground state  are uniform. 
\end{abstract}
\bigskip

\bigskip


\end{titlepage}

\section{Introduction}
The
first mathematical results on the stability of matter were obtained 
by Dyson and Lenard 
in \ci{D1967, DL1968} 
where the energy bound from below
was established.
The thermodynamic limit
for the Coulomb systems  
was studied first by Lebowitz and Lieb 
\ci{LL1969,LL1973}, see the survey and further development in \ci{LS2010}.
These results were extended  by Catto, Le Bris, Lions 
and others to Thomas-Fermi and Hartree-Fock models
\ci{CBL1998,CBL2001,CBL2002}. 
Further results in this direction were established 
by Canc\'es,  Lahbabi, Lewin, Sabin, Stoltz, and others
 \ci{CLL2013,CS2012, BL2005, LS2014-1, LS2014-2}.
All these results concern  either
the convergence of
the ground state of finite particle systems
in the thermodynamic limit or 
the existence of the ground state
for infinite particle systems. 

However,
the dynamical stability of crystals 
with moving ions
was never considered previously. 
This stability is 
necessary for a rigorous analysis 
of fundamental quantum phenomena in the solid state physics: 
heat conductivity, electric conductivity, thermoelectronic emission, photoelectric effect, 
Compton effect, 
etc., see \ci{BLR}.

In present paper we
consider the coupled
Schr\"odinger-Poisson-Newton equations for finite  crystals
under periodic boundary conditions with one ion per cell of a lattice.
We
construct the global dynamics   of crystals
with  moving ions and prove the conservation of energy and charge.

Our main result
is  the  orbital stability of every 
 ground state with periodic arrangement of ions
under novel `Jellium' and Wiener-type  conditions on the ion charge density.

The electron cloud is described by  one-particle 
Schr\"odinger equation.
The ions are described as classical particles   
that corresponds to the  
Born and Oppenheimer  approximation.
The ions interact with the electron cloud via 
the scalar potential, which  is a solution to the corresponding Poisson equation.

This model
does not respect the Pauli exclusion principle for electrons.
However, it provides a convenient framework 
to introduce suitable functional tools that  might be useful 
for physically 
more realistic models (Thomas-Fermi, Hartree-Fock, and second quantized models). 
In particular, we find a novel 
stability criterion (\re{Wai}), (\re{W1}).
\medskip

We consider crystals which occupy the finite torus ${T_N}:=\R^3/N\Z^3$
and have one ion per cell of the cubic lattice ${\Gamma_N}:=\Z^3/N\Z^3$, where $N\in\N$.
The cubic lattice  is chosen for the simplicity of notations.
We denote by  $\sigma(x)$ the charge density of one ion,
\begin{equation}\la{ro+}
\si\in C^2({T_N}),\qquad
\int_{{T_N}} \sigma(x)dx=eZ>0, 
\end{equation} 
where $e>0$ is the elementary charge.
Let $\psi(x,t)$ be the wave function of the electron field, 
$q(n,t)$ denotes the ion displacement  from the reference position $n\in\Gamma_N$,
and $\Phi(x)$ be the electrostatic  potential generated by the ions and electrons.
We assume $\hbar=c=\cm=1$, where $c$ is the speed of light and $\cm$ is the electron mass.
Then the considered coupled equations  read
\begin{eqnarray}\la{LPS1}
i\pa_t\psi(x,t)\!\!&=&\!\!-\fr12\De\psi(x,t)-e\Phi(x,t)\psi(x,t),\qquad x\in{T_N},
\\
\nonumber\\
-\De\Phi(x,t)\!\!&=&\!\!\rho(x,t):=\sum_{n\in{\Gamma_N}}
\sigma(x-n-q(n,t))-e|\psi(x,t)|^2,\qquad x\in{T_N},
\la{LPS2}
\\
\nonumber\\
M\ddot q(n,t)
\!\!&=&\!\!-(\na\Phi(x,t),\sigma(x-n-q(n,t))), 
\qquad n\in{\Gamma_N}.
\la{LPS3}
\end{eqnarray}
Here the 
brackets $(\cdot,\cdot)$
 stand for the  scalar product on the real Hilbert
space $L^2({T_N})$ and for its different extensions,  and $M>0$ is the mass of one ion.
All derivatives here and below are understood in the sense of distributions.
Similar finite periodic approximations of crystals are treated in all textbooks on 
quantum theory of solid state \ci{Born, Kit, Zim}. 
However,  the stability of ground states in this model was never discussed.

Obviously, 
\begin{equation}\la{r0}
\ds\int_{T_N}\rho(x,t)dx=0 
\end{equation}
by the  Poisson equation (\re{LPS2}).
Hence, the potential $\Phi(x,t)$ can be eliminated from  the system (\re{LPS1})\,-\,(\re{LPS3})
using the operator $G:=(-\De)^{-1}$, see (\re{fs}) for a more precise definition. 
Substituting $\Phi(\cdot,t)=G\rho(\cdot,t)$
into equations (\ref{LPS1}) and (\ref{LPS3}), we can write the system as
\begin{equation}\la{vf}
\dot X(t)=F(X(t)),\qquad t\in\R,
\end{equation}
where $X(t)=(\psi(\cdot,t), q(\cdot,t), p(\cdot,t))$ with $p(\cdot,t):=\dot q(\cdot,t)$.
The system (\ref{LPS1})\,-\,(\ref{LPS3}) is equivalent,  up to a  gauge transform (see the next section), 
to equation (\ref{vf}) with the normalization
\begin{equation}\la{rQ}
\Vert\psi(\cdot,t)\Vert_{L^2({T_N})}^2=ZN^3,\qquad t\in\R,
\end{equation}
which follows from (\re{r0}). If the integral (\ref{ro+}) vanishes, 
we have $Z=0$ and  $\psi(x,t)\equiv 0$.

We will identify the complex  functions $\psi(x)$ with 
two real functions $\psi_1(x):=\rRe\psi(x) $ 
and  $\psi_2(x):=\rIm\psi(x)$.
Now
equation (\ref{vf}) can be written as the Hamilton system
\begin{equation}\la{HSi}
\pa_t \psi_1(x,t)=\fr12 \pa_{\psi_2(x)}E,~~\pa_t \psi_2(x,t)=-\fr12\pa_{\psi_1(x)}E,~~
\pa_t q(n,t)= \pa_{p(n)}E,~~\pa_t p(n,t)=-\pa_{q(n)} E.
\end{equation}
Here  the Hamilton functional (energy) reads 
\begin{equation}\la{Hfor}
  E(\psi, q, p)=\fr12\int_{{T_N}}|\na\psi(x)|^2dx+\fr12(\rho,G\rho)+\sum_{n\in{\Gamma_N}} \fr{p^2(n)}{2M},
\end{equation}
where  $ q:=(q(n): ~~n\in{\Gamma_N})\in[{T_N}]^{\ov N}$, $ p:=(p(n):~~n\in{\Gamma_N})\in\R^{3\ov N}$ with $\ov N:=N^3$, and 
\begin{equation}\la{Hfor2}
\rho(x):=
\sum_{n\in{\Gamma_N}}\si(x-n-q(n))-e|\psi(x)|^2,\qquad x\in{T_N}.
\end{equation}

Our main goal is the stability of 
ground states, i.e.,
solutions to (\re{LPS1}) \,-\, (\re{LPS3})
with minimal (zero) energy  (\re{Hfor}).
We will consider only  ${\Gamma_N}$-periodic ground states 
 (nonperiodic ground states exist for some degenerate 
densities $\si$, 
see Remark \re{r1} ii) and Section \re{snp}).

We will see that all these ${\Gamma_N}$-periodic ground states  can be stable  
depending on the choice of the ion density $\sigma$.
However,
we study very special densities 
$\sigma$ satisfying some conditions below. Namely,
we will assume
 the following condition on the ion charge density,
\begin{equation}\la{Wai}
 \mbox{\bf The Jellium Condition:}~~~~~ \hat\si(\xi)
  :=\int_{T_N} e^{i\xi x}\si(x)dx
 =0,\quad \xi\in {\Gamma^*_1}\setminus 0,
~~~~~~~~~~~~ ~~~~~~~~~~~~~~~~~~ ~
\end{equation}
where ${\Gamma^*_1}:=2\pi\Z^3$.
This condition immediately implies that the periodized ion charge density is a positive 
constant everywhere on the torus:
\begin{equation}\la{sipi}
	\sum_{n\in{\Gamma_N}}\si(x-n)\equiv eZ,\qquad x\in{T_N}.
\end{equation}
 The simplest example of such a 
$\sigma$ is a constant over the unit cell of a given lattice, which is what physicists 
usually call Jellium \cite{GV2005}. We give further examples in Section \re{sex}.
Here we study this model in the rigorous context of the Schr\"odinger-Poisson equations.

Furthermore, we will assume a spectral property of the Wiener type
\begin{equation}\la{W1}
\mbox{\bf The Wiener Condition:}~~~\Si(\theta):=\sum_{m\in\Z^3}\Big[
 \fr{\xi\otimes\xi}{|\xi|^2}|\hat\si(\xi)|^2\Big]_{\xi=\theta+2\pi m}>0,\
\quad \theta\in \Pi^*_N\setminus {\Gamma^*_1},
\end{equation}
where the Brillouin zone $\Pi^*_N$ is defined by
\begin{equation}\la{PPG}
 \Pi^*_N:= \{\xi=(\xi^1,\xi^2,\xi^3)\in{\Gamma^*_N}:0\le \xi^j\le 
{\color{red}  2\pi},~~j=1,2,3\},\quad{\Gamma^*_N}:=\fr{2\pi}N\Z^3.
\end{equation}
This condition is
 an analog of the Fermi Golden Rule
for crystals. 
It is independent of (\re{Wai}).
 We have introduced  
conditions of type (\re{Wai}) and  (\re{W1}) in \ci{KKpl2015} 
 in the framework of infinite crystals.
\br\la{rW}
{\rm
i) The series \eqref{W1} converges for $\theta\in{\Gamma^*_N}\setminus{\Gamma^*_1}$
by the Parseval identity since $\sigma\in L^2(T_N)$ by \eqref{ro+}.
\\
ii) The matrix $\Si(\theta)$ is $\Gamma^*_1$-periodic outside  $\Gamma^*_1$.
Thus, (\re{W1}) means that $\Si(\theta)$ is a positive matrix
for $\theta\in \ov\Pi^*_N\setminus 0$, where $\ov\Pi^*_N$
is the `discrete torus' $\Gamma_N/\Gamma^*_1$.
}
\er
The series \eqref{W1} is a nonnegative matrix.
Hence,  the Wiener condition holds `generically'. For example it holds if
\begin{equation}\la{W1s}
\hat\si(\xi)\ne 0,\qquad \xi\in {\Gamma^*_N}\setminus{\Gamma^*_1},
\end{equation}
i.e., (\re{Wai}) are the only zeros of $\hat\si(\xi)$. However,  (\re{W1}) does not hold 
for the simplest Jellium model, when $\sigma$ is constant on the unit cell, see (\re{sic}) and (\re{sJM}). 

The energy (\re{Hfor}) is nonnegative, and its minimum is zero.
We show in Lemma \re{Jgs} that under Jellium condition (\re{Wai}) all ${\Gamma_N}$-periodic
ground states are zero energy stationary solutions of the form
\begin{equation}\la{gr}
S_{\al,r}=(\psi_\al, \ov r,0),\qquad \al\in [0,2\pi], \quad r\in{T_N},
\end{equation}
where $\psi_\al(x)\equiv e^{i\al}\sqrt{Z}$ and $\ov r\in [{T_N}]^{\ov N}$  is defined by
\begin{equation}\la{gr2}
 \ov r(n)=r,\qquad n\in{\Gamma_N}.
\end{equation}
The corresponding electronic charge  density reads
\begin{equation}\label{roZ}
 \rho^e(x):=-e|\psi_\al(x)|^2\equiv -eZ,\qquad x\in {T_N}.
 \end{equation}
Hence,
the corresponding total charge density (\re{Hfor2}) identically vanishes by (\re{sipi}). 
Let us emphasize that 
 both ionic and  electronic charge densities  are uniform for the ground state under the Jellium condition.

\smallskip

Our main result (Theorem \re{tm}) is the stability  
of the real 4-dimensional `solitary manifold'
\begin{equation}\la{cS}
\cS=\{S_{\al ,r}: 
~\al \in [0,2\pi],~r\in{T_N} \}.
\end{equation}
The stability means that any solution $X(t)=(\psi(\cdot,t), q(\cdot,t), p(\cdot,t))$ 
to (\re{vf})
 with initial data, 
lying in the vicinity of
the manifold $\cS$, is close to it uniformly in time. 
This is the `orbital stability' in the sense of \ci{GSS87}, since the manifold
$\cS=S^1\times{T_N}\times \{0\}$ is the orbit of the symmetry group $U(1)\times{T_N}$.
 Obviously, 
\begin{equation}\la{ES}
E(S)= 0,\qquad S\in\cS.
\end{equation}

Let us comment on our approach.
We prove the local well-posedness for the system  (\re{LPS1})\,-\,(\re{LPS3}) 
by
the contraction mapping principle.
The global  well-posedness for the equation (\re{vf})
and the charge and energy conservation
follow by the Galerkin approximations and the uniqueness of solutions. 
We apply  the charge 
conservation to return back to the system  (\re{LPS1})\,-\,(\re{LPS3}).

The orbital stability of the solitary manifold $\cS$ 
is deduced
from  the lower energy estimate
\begin{equation}\la{BLi}
E(X)\ge \nu\,d^2(X,\cS)\qquad{\rm if}\qquad d(X,\cS)\le \delta,\quad X\in\cM,
\end{equation}
where 
$\cM$ is the manifold defined by  the normalization  (\re{rQ}) (see Definition \re{dM});
 $\nu,\delta>0$ and `$d$' is the distance in the `energy norm'.
This estimate obviously implies  the stability of the solitary manifold $\cS$.
We deduce  (\re{BLi}) from the positivity of the Hessian $E''(S)$
for $S\in \cS$ in the orthogonal directions to $\cS$ on the manifold $\cM$.
The Jellium and Wiener conditions are sufficient 
for this positivity.
We expect that these conditions are also necessary;
however, this is still an open  problem.
Anyway,
the positivity can break
down when these conditions  fail. We have shown this in \ci[Lemma 10.1]{KKpl2015}
in the context of infinite crystals, however the proof extends directly to the finite
crystals.
The Jellium condition  cancels the negative
energy which is provided by the electrostatic instability
 (`Earnshaw's Theorem' \ci{Stratton}, 
see \ci[Remark 10.2]{KKpl2015}).
\medskip
 
Our main novelties are the following.
\smallskip\\
I. The well-posedness and the energy and charge conservation for 
the system  (\re{LPS1})\,-\,(\re{LPS3}).
\smallskip\\
II.  The 
calculation of  all ground states; in particular, the existence of 
ground states with periodic and with
non-periodic  ion arrangements.
\smallskip\\
III. The orbital stability 
of ${\Gamma_N}$-periodic ground states.
\smallskip\\
IV. The lower energy estimate (\re{BLi}).
\begin{remarks}\label{r1}
{\rm
 i) In the case of infinite crystal, corresponding to $N=\infty$, the orbital stability seems  impossible. 
 Namely, for  $N=\infty$ the estimates (\ref{eq}), (\ref{GP2}) and (\ref{fp}) break down,
 as well as the estimate of type (\ref{BLi}) which is due to the discrete spectrum of the energy
 Hessian $E''(S)$ on the compact torus.
\smallskip\\
ii) We  show that the identity of type 
(\re{sipi}) holds  for a wide set of  
arrangements of ions  which are not $\Gamma_1$-periodic,
if $\si$ satisfy additional spectral conditions.
The corresponding examples are given, but in all our examples 
 the Wiener condition breaks down. We suppose that the Wiener condition provides
the periodicity (\re{gr2}), however this is a challenging open problem, 
see Section \re{spp}.  We prove the orbital stability only  for $\Gamma_1$-periodic ground states.
\smallskip\\
iii)  The extension of our results to the  Hartree-Fock model is not straightforward.
Even the existance of solutions requires quite novel ideas as well as the calculation of the null space
of the Hessian.
}
\end{remarks}

Let us comment on previous works in this direction.
\smallskip\\
 The ground state for crystals 
in the Schr\"odinger-Poisson model was constructed in 
\ci{K2014,K2015}, and its linear stability was proved in \ci{KKpl2015}.

In the Hartree-Fock model
the crystal ground state 
 was constructed for the first time by Catto, Le Bris, and  Lions  \ci{CBL2001,CBL2002}.
For the Thomas-Fermi model, see \ci{CBL1998}.

In \ci{CS2012}, Canc\'es and Stoltz have established the well-posedness  for 
the dynamics of  
local perturbations of the crystal ground state 
in the  {\it random phase approximation}
for the reduced  Hartree-Fock equations
with the Coulomb  pairwise interaction potential $w(x-y)=1/|x-y|$.
The  space-periodic nuclear potential
in the equation \ci[(3)]{CS2012}
does not depend on time that corresponds to 
the fixed nuclei positions. 
The nonlinear Hartree-Fock dynamics
for crystals
with the Coulomb potential and
without the  random phase approximation
was not studied previously,
see the discussion in 
\ci{BL2005} and in the introductions of \ci{CLL2013,CS2012}.

In \ci{CLL2013} 
E. Canc\`es, S. Lahbabi, and M. Lewin have considered
 the random reduced HF model of crystal  when 
the ions charge density and the electron density matrix are random processes,
and the action of the lattice translations on the probability space is ergodic.
The authors obtain suitable generalizations of the Hoffmann-Ostenhof 
and Lieb-Thirring inequalities  for ergodic density matrices, 
and
construct a random potential which is a solution  to 
 the Poisson equation 
with the corresponding stationary stochastic  charge density. 
The main result is the  coincidence of this model with the thermodynamic limit in  
the case of the short range Yukawa interaction.

In \ci{LS2014-1}, Lewin and Sabin have established the well-posedness for the 
reduced von Neumann equation, describing the Fermi gas,
 with density matrices of infinite trace 
and pair-wise interaction potentials $w\in L^1(\R^3)$. Moreover, the authors  
prove the asymptotic stability of translation-invariant stationary states 
for 2D Fermi gas \ci{LS2014-2}.
\smallskip

 The paper is organized as follows.
 In Section 2 we eliminate the potential and reduce the dynamics  to the integral equation.
In Sections 3 
we prove the well-posedness.
In Section 4  we prove 
the stability of the solitary manifold $\cS$ establishing  
the lower estimate for the energy. 
In Appendices we 
prove the conservation
of the energy and charge, 
describe all ground states and give some examples.
\medskip\\
{\bf Acknowledgments} The authors are grateful to Herbert Spohn for helpful 
discussions and remarks.


\setcounter{equation}{0}
\section{Reduction to the integral equation}
The operator $G:=(-\De)^{-1}$   is well defined in the Fourier series:
\begin{equation}\la{fs}
\rho(x)=\sum_{\xi\in{\Gamma^*_N}}\hat\rho(\xi)e^{i\xi x},
\qquad G\rho:=\sum_{\xi\in{\Gamma^*_N}\setminus 0}\fr{\hat\rho(\xi)}{\xi^2}e^{i\xi x}.
\end{equation}
 The Poisson equation (\ref{LPS2})
implies that $\hat\rho(0,t)=\ds\int\rho(x,t)\,dx=0$, which is  equivalent to (\re{rQ}). 
 Hence,
$\Phi(\cdot,t)=G\rho(\cdot,t)$ up to an additive constant $C(t)$ which can be compensated by a gauge transform 
$\psi(x,t)\mapsto\psi(x,t)\exp(ie\ds\int_0^t C(s)ds)$. 
The system (\re{HSi}) 
can be written as 
\begin{equation}\la{HS}
\dot X(t)=JE'(X(t)),
\qquad X(t):=(\psi_1(t)+i\psi_2(t),  q(t),p(t)),
\end{equation}
where 
\begin{equation}\la{HS2}
J=\left(
\begin{array}{rrr}
-i/2 &0&0\\
0   &0&1\\
0   &-1&0
\end{array}
\right).
\end{equation}
We will use the following function spaces with $s=0,\pm 1$.
Let us define the Sobolev space $H^s({T_N})$ as real Hilbert spaces of complex-valued functions
with the scalar product 
\begin{equation}\la{sp}
(\psi,\vp)_s:=\rRe\int_{T_N}\sum_{|\al|\le s} \pa^\al\psi(x)\pa^\al\ov\vp(x)dx,\qquad s=0,1.
\end{equation}
By definition, $H^{-1}({T_N})$ is the real dual space to $H^1({T_N})$
which will be identified with distributions by means of the scalar product in $H^0({T_N})$.

\bd
i) $\cW^s$ denotes the  real Hilbert space 
$H^s({T_N})\oplus\R^{3\ov N}\oplus \R^{3\ov N}$ for $s=0,\pm1$.
\medskip\\
ii) $\cV^s:=
H^s({T_N})\times[{T_N}]^{\ov N}\times \R^{3\ov N}$ is 
the Hilbert manifold endowed with the  metric 
\begin{equation}\la{dVs}
d_{\cV^s}(X,X'):=\Vert\psi-\psi'\Vert_{H^s({T_N})}+|q-q'|+|p-p'|,\qquad X=(\psi,q,p),
\quad
 X'=(\psi',q',p')
\end{equation}
and with the `quasinorm'
\begin{equation}\la{cVs}
|X|_{\cV^s}:=\Vert\psi\Vert_{H^s({T_N})}+|p|,\qquad X=(\psi,q,p).
\end{equation}

\ed

The linear space $\cW^s$ is  isomorphic to the tangent space to 
the Hilbert manifold $\cV^s$ at each point $X\in\cV^s$. 
We will write $\cX:=\cV^0$, $\cV:=\cV^1$, $\cW:=\cW^1$,
and $(\cdot,\cdot)_0=(\cdot,\cdot)$, which agrees with the definition of  
the scalar product on the real Hilbert space $L^2(T_N)$. In particular,
\begin{equation}\la{1i}
(1,i)=0.
\end{equation}
Denote 
by the brackets
$\langle\cdot,\cdot\rangle$
the scalar product in $\cX$ and also the duality between 
$\cW^{-1}$ and $\cW^1$:
\begin{equation}\la{dWW}
\langle Y,Y'\rangle:=(\vp,\vp')+\vka\vka'+\pi\pi',\qquad Y=(\vp,\vka,\pi), 
\quad Y'=(\vp',\vka',\pi').
\end{equation}
The total electronic charge  is defined (up to a factor) by
\begin{equation}\la{Q}
Q(X):= \int|\psi(x)|^2dx,\qquad X=(\psi,q,p)\in\cV.
\end{equation}
Obviously,
\begin{equation}\la{EQV}
|X|_\cV^2\le C[E(X)+Q(X)],\qquad X\in\cV,
\end{equation}
The system 
(\re{HS})
is a  nonlinear infinite-dimensional perturbation of the free Schr\"odinger equation.
We will prove  that a 
solution $X\in C(\R,\cV)$ exists and is unique for any initial state $X(0)\in\cV$, and the
energy and the electronic charge are conserved,
\begin{equation}\la{EQ}
E(X(t))=E(X(0)),\quad Q(X(t))=Q(X(0)),\qquad t\in\R.
\end{equation}
The energy (\re{Hfor}) and the charge are well defined and continuous 
on $\cV$ in the metric $d_\cV$ by the estimate (\re{Th}) below.
The charge conservation holds 
by the Noether theory \ci{A, GSS87, KQ} due to the $U(1)$-invariance of the Hamilton functional:
\begin{equation}\la{U1}
E(e^{i\al }\psi,q,p)=E(\psi,q,p),\qquad (\psi,q,p)\in\cV,\quad \al \in\R.
\end{equation}
We rewrite the system (\ref{HS}) in the integral  form
\begin{equation}\la{LPSi}
\left\{\begin{array}{lll}
\psi(t)&=&e^{-iH_0t} \psi(0)+ie\ds\int_0^t e^{-iH_0(t-s)} [ \Phi(s)\psi(s) ]ds,\\
q(n,t)&=&q(n,0)+\frac 1M\ds\int_0^t p(n,s)ds\mod N\Z^3,\\
p(n,t)&=&p(n,0)-\ds\int_0^t (\nabla \Phi(s),\sigma(\cdot-n-q(n,s))) ds,
\end{array}\right|
\end{equation}
where $H_0:=-\fr12\De$ and $\Phi(s):=G\rho(s)$.
In the vector form (\ref{LPSi}) reads
\begin{equation}\la{LPSiv}
X(t)=e^{-At}X(0)+\int_0^t  e^{-A(t-s)} N(X(s)) ds\mod \left(\begin{array}{c}0\\ N\Z^3\\0\end{array}\right).
\end{equation}
Here
\begin{equation}\la{HN}
A=\left(\begin{array}{ccc}
iH_0 & 0 & 0\\
0&0&0\\
0&0&0
\end{array}\right),\quad
N(X)=(ie \Phi\psi ~, p,~f),\quad 
f(n):=-(\nabla \Phi,\sigma(\cdot-n-q(n))),\quad\Phi:=G\rho,
\end{equation}
where $\rho$ is defined by (\re{Hfor2}).

\setcounter{equation}{0}
\section{Global dynamics}\la{Gd}
In this section we prove the well-posedness of the dynamics.

\bt\label{TLWP1}(Global well-posedness).
Let   (\re{ro+}) hold and $X(0)\in\cV$. Then 
\medskip\\
i) 
 Equation (\ref{HS}) has  a unique  solution $X\in C(\R,{\cal V})$,
 and the maps $U(t):X(0)\mapsto X(t)$ are continuous in $\cV$ for $t\in\R$.
 \medskip\\
ii) The conservation laws (\re{EQ}) hold.
\medskip\\
iii) $X$ is the solution to (\ref{LPS1})\,-\,(\ref{LPS3}) if 
\begin{equation}\la{rQ2}
Q(X(0))=Z\ov N^3.
\end{equation}
\et

First,  let us  prove the local well-posedness.

\bp\label{TLWP}(Local well-posedness).
Let   (\re{ro+}) hold and $|X(0)|_\cV\le R$. Then 
there exists $\tau=\tau(R)>0$ such that
  equation (\ref{HS}) has  a unique  solution $X\in C([-\tau,\tau],{\cal V})$,
 and the maps $U(t):X(0)\mapsto X(t)$ are continuous in $\cV$ for $t\in [-\tau,\tau]$.

 \ep
In the next two lemmas
we prove
the boundedness and the local Lipschitz continuity of the nonlinearity $N:\cV\to\cW$.
With this proviso Proposition \re{TLWP} follows from the integral form 
(\re{LPSiv}) of the equation (\ref{HS})
 by the contraction mapping principle, since {\color{red} $e^{-At}$} is an isometry of $\cW$.
\begin{lemma}\label{p1}
For any $R>0$ and $|X|_\cV\le R$
\begin{equation}\label{bN}
\Vert N(X)\Vert_\cW\le C(R)
\end{equation}
\end{lemma}
\begin{proof}
We split $\rho(x,t)$ as
  $\rho(x,t)=\rho^i(x,t)+\rho^e(x,t)$,
where
\[
\rho^i(x,t)=\sum_{n\in{\Gamma_N}}\sigma(x-n-q(n,t)),\qquad \rho^e(x,t)=-e|\psi(x,t)|^2.
\]
Applying the Cauchy-Schwarz inequality to the second formula (\ref{fs}), 
we obtain for $\Phi:=G\rho$,
\begin{equation}\label{Gpsi}
\Vert \Phi\Vert_{C({T_N})}\le C\Vert\hat \rho\Vert_{L^2({\Gamma^*_N})}=C\Vert \rho\Vert_{L^2({T_N})}
\le C(\Vert \rho^i\Vert_{L^2({T_N})}+e\Vert \psi\Vert_{L^4({T_N})}^2)\le
 C_1 (1+\Vert\psi\Vert_{H^1({T_N})}^2)
\end{equation}
since $H^1({T_N})\subset L^6({T_N})\subset L^4({T_N})$.
On the other hand, the H\"older inequality implies that
\begin{equation}\label{np}
\Vert\nabla \rho^e\Vert_{L^{3/2}({T_N})}\le e
\Vert\na|\psi|^2\Vert_{L^{3/2}({T_N})}
\le C_1\Vert\psi\Vert_{L^6({T_N})}\Vert\nabla\psi\Vert_{L^2({T_N})}
\le C_2\Vert\psi\Vert_{H^1({T_N})}^2.
\end{equation}
Therefore, we get by the Hausdorff-Young and the H\"older inequalities {\color{red} \ci{Her1}}
\begin{eqnarray}\nonumber
\Vert \nabla \Phi\Vert_{L^3({T_N})}\!\!&\le&\!\! C\Vert\widehat{\nabla \Phi}\Vert_{L^{3/2}({\Gamma^*_N})}
\le C_1\Vert \xi\hat \rho\Vert_{L^3({\Gamma^*_N})}
\Big[\sum_{\xi\in{\Gamma^*_N}\setminus 0}|\xi|^{-6}\Big]^{1/3}
\le C_2\Vert\nabla \rho\Vert_{L^{3/2}({T_N})}\\
\label{eq}
\!\!&\le&\!\! C_2(\Vert\nabla \rho^{i}\Vert_{L^{3/2}({T_N})}+\Vert\nabla \rho^{e}\Vert_{L^{3/2}({T_N})})
\le C_3(1+\Vert\psi\Vert_{H^1({T_N})}^2).
\end{eqnarray}
Now  (\ref{Gpsi})  and  (\ref{eq}) imply by the H\"older inequality 
\begin{eqnarray}\nonumber
\Vert\psi \Phi\Vert_{L^2({T_N})}&\le& \Vert \Phi\Vert_{C({T_N})}\cdot\Vert\psi \Vert_{L^2({T_N})}
\le C(1+\Vert\psi\Vert_{H^1({T_N})}^3)\\
\nonumber
\Vert \nabla\psi\Phi\Vert_{L^2({T_N})}&\le& \Vert \Phi\Vert_{C({T_N})}\Vert\nabla\psi \Vert_{L^2({T_N})}\le C(1+\Vert\psi\Vert_{H^1({T_N})}^3)\\
\Vert \psi\nabla \Phi \Vert_{L^2({T_N})}
&\le& C\Vert\psi\Vert_{L^6({T_N})}\cdot \Vert\nabla \Phi\Vert_{L^3({T_N})}
\le C_1(1+\Vert\psi\Vert_{H^1({T_N})}^3).\nonumber
\end{eqnarray}
Hence,
\begin{equation}\label{eq2}
\Vert\Phi\psi \Vert_{H^1({T_N})}\le C(1+\Vert\psi\Vert_{H^1({T_N})}^3).
\end{equation}
Finally, (\ref{eq}) and (\ref{ro+}) imply that
\begin{equation}\label{eq3}
|f(n)|\le \Vert \Phi\Vert_{C({T_N})}\Vert\na\si\Vert_{L^1({T_N})}
\le C(1+\Vert\psi\Vert_{H^1({T_N})}^2),\qquad n\in{\Gamma_N}.
\end{equation}
At last, (\ref{bN}) holds by  (\ref{eq2}) and  (\ref{eq3}).
\end{proof}
It remains
to prove  that the nonlinearity is  locally Lipschitz.

\begin{lemma}\label{p2}
For any $R>0$ and $X_1,X_2\in\cV$ 
\begin{equation}\label{lN}
\Vert N(X_1)-N(X_2)\Vert_\cW\le C'(R) d_\cV(X_1, X_2)\qquad{\rm if}\quad
| X_1|_\cV, |X_2|_\cV\le R.
\end{equation}
\end{lemma}

\begin{proof}
Writing  $X_k=(\psi_k,q_k,p_k)$ and $\Phi_k=G\rho_k$, we obtain that
\begin{equation}\label{NN1}
\Vert\Phi_1\psi_1-\Phi_2\psi_2\Vert_{H^1({T_N})}
\le
\Vert (\Phi_1-\Phi_2)\psi_1\Vert_{H^1({T_N})}
+\Vert \Phi_2(\psi_1-\psi_2)\Vert_{H^1({T_N})}.
\end{equation}
Similarly to (\ref{Gpsi})\,-\,(\ref{eq}) we obtain
\begin{eqnarray}\nonumber
\Vert \Phi_2(\psi_1-\psi_2)\Vert_{H^1({T_N})}&\le&
\Vert \Phi_2\Vert_{C({T_N})}
\Vert\psi_1-\psi_2\Vert_{H^1({T_N})}+\Vert \nabla \Phi_2\Vert_{L^3({T_N})}
\Vert\psi_1-\psi_2\Vert_{L^6({T_N})}\\
\label{NN2}
&\le&C(1+R^2)\Vert\psi_1-\psi_2\Vert_{H^1({T_N})}\le C(R)d_\cV(X_1,X_2)
\end{eqnarray}
Further, similarly to (\ref{np}),
\begin{equation}\label{np2}
\Vert\nabla (\rho^e_1-\rho^e_2)\Vert_{L^{3/2}({T_N})}
\le C\Vert\psi_1-\psi_2\Vert_{H^1({T_N})}[\Vert\psi_1\Vert_{H^1({T_N})}+\Vert\psi_2\Vert_{H^1({T_N})}].
\end{equation}
Moreover, $|\si(x)-\si(x-a)|\le C|a|$,
where $|a|:=\min_{r\in a} |r|$ for $a\in T_N$.
Hence, similarly to (\ref{eq}),
\begin{align*}
\Vert(\Phi_1-\Phi_2)\psi_1\Vert_{H^1({T_N})}\le
\Vert \Phi_1-\Phi_2\Vert_{C({T_N})}
\Vert\psi_1\Vert_{H^1({T_N})}+
\Vert \nabla (\Phi_1-\Phi_2)\Vert_{L^3({T_N})}
\Vert\psi_1\Vert_{L^6({T_N})}
\\
\le 
C R\Big[\Vert \rho_1^i-\rho_2^i\Vert_{L^2({T_N})}+\Vert \rho_1^e-\rho_2^e\Vert_{L^2({T_N})}
+\Vert\nabla(\rho_1^i-\rho_2^i)\Vert_{L^{3/2}({T_N})}+\Vert\nabla(\rho_1^e-\rho_2^e)\Vert_{L^{3/2}({T_N})}\Big]
\\
\le C_1 R(|q_1-q_2|+R\Vert \psi_1-\psi_2\Vert_{H^1({T_N})})
\le C(R)d_\cV(X_1,X_2).
\end{align*}
Now (\re{NN1}) and  (\re{NN2}) give
\begin{equation}\la{Th}
\Vert\Phi_1\psi_1-\Phi_2\psi_2\Vert_{H^1({T_N})}\le C(R)
d_\cV(X_1,X_2).
\end{equation}
Similarly,
\begin{eqnarray}\nonumber
&&| (\nabla \Phi_1,\sigma(\cdot-n-q_1(n)))- (\nabla \Phi_2,\sigma(\cdot-n-q_2(n)))|\\
\nonumber
&&\le |(\nabla \Phi_1-\nabla\Phi_2,\sigma(\cdot-n-q_1(n)))|
+ |(\nabla \Phi_2,\sigma(\cdot-n-q_1(n))-\sigma(\cdot-n-q_2(n)))|\\
\nonumber
&&\le C(\Vert \Phi_1-\Phi_2\Vert_{C({T_N})}+\Vert \Phi_2\Vert_{C({T_N})}| q_1-q_2|\le C(R)
d_\cV(X_1,X_2).
\end{eqnarray}
This estimate and  (\re{Th}) imply (\re{lN}).
\end{proof}
Now Proposition \re{TLWP} follows from Lemmas \ref{p1} and \ref{p2}. 
\medskip\\
{\bf Proof of Theorem \re{TLWP1}.}
The local solution  $X\in C([-\tau,\tau],\cV)$  exists and is unique by Proposition \re{TLWP}.
On the other hand, the conservation laws (\ref{EQ}) (proved in Proposition \re{lgal} iii)) 
together with   (\re{EQV}) imply a priori bound
\begin{equation}\la{ab}
| X(t)|_\cV^2\le C[E(X(0))+ Q(X(0))],\qquad t\in [-\tau,\tau].
\end{equation}
Hence, the local solution admits an extension to the global one $X\in C(\R,\cV)$. 
Further, (\ref{rQ2}) implies that $Q(X(t))=Z\ov N^3$ for all $t\in\R$ by the charge conservation 
(\re{EQ}). Hence, (\ref{HS}) gives (\ref{LPS1})\,-\,(\ref{LPS3}).
\hfill$\Box$
\setcounter{equation}{0}
\section{The orbital stability of the ground state}
In this section  we expand the energy into the Taylor series and prove the orbital stability 
checking the  positivity of the energy Hessian.
\subsection{The Taylor expansion of the Hamilton functional}

We will deduce the lower estimate (\re{BLi}) using
 the  Taylor expansion of $E(S+Y)$ for 
$S=S_{\al,r}:=(\psi_\al,\ov r,0)\in \cS$ and $Y=(\vp,\vka,p)\in\cW$:
\begin{equation}\la{te} 
E(S+Y)=E(S)+\langle E'(S),Y\rangle +\fr12 \langle Y,E''(S)Y\rangle + R(S,Y)=
\fr12 \langle Y,E''(S)Y\rangle + R(S,Y)
\end{equation}
since $E(S)=0$ and $E'(S)=0$.
First,
we expand 
 the  charge density (\re{Hfor2})
corresponding to 
$S+Y=(\psi_\al+\vp,\ov r+\vka,p)$:
\begin{equation}\la{ro}
\rho(x)=\rho^{(0)}(x)+\rho^{(1)}(x)+\rho^{(2)}(x),\qquad x\in{T_N},
\end{equation}
where $\rho^{(0)}$ and  $\rho^{(1)}$ are respectively 
the  terms of zero and first  order in 
$Y$, while  $\rho^{(2)}$ is the remainder.
However, $\rho^{(0)}(x)$ is the total charge density of the ground state
which is identically zero by (\ref{sipi}) and (\re{roZ}):
\begin{equation}\la{ro0}
\rho^{(0)}(x)=\rho^i_0(x)-e|\psi_\al(x)|^2\equiv 0,\qquad x\in{T_N}.
\end{equation}
Thus, $\rho=\rho^{(1)}+\rho^{(2)}$.
Expanding (\re{Hfor2}) further, we obtain 
\begin{eqnarray}
\la{ro11}\rho^{(1)}(x)\!\!\!\!&\!\!\!\!=\!\!\!\!&\!\!
\si^{(1)}(x)-2e\psi_\al (x)\cdot\vp(x),\quad \si^{(1)}(x)=-\sum_{n\in{\Gamma_N}} \vka(n)\cdot\na\si(x-n-r),\quad
\\
\la{ro13}\rho^{(2)}(x)\!\!\!\!&\!\!=\!\!\!\!&\!\!\si^{(2)}(x)-e|\vp(x)|^2,~~ \si^{(2)}(x)=\fr12\sum_{n\in{\Gamma_N}}
\int_0^1\!\!(1\!-\!s)
[\vka(n)\cdot\na]^2\si(x-n-r-s\vka(n))ds.
\end{eqnarray}
Substituting  $\psi=\psi_\al +\vp$ and $\rho=\rho^{(1)}+\rho^{(2)}$  
into (\ref{Hfor}), we obtain that
the  quadratic part  of (\re{te}) reads 
\begin{equation}\la{B2}
\fr12\langle Y,E''(S) Y \rangle=\fr12\int_{{T_N}}|\na \vp(x)|^2dx+
\fr12 (\rho^{(1)},G\rho^{(1)})+K(p),~~~~~~ K(p):=\ds\sum_n\fr{p^2(n)}{2M}~~
\end{equation}
and the remainder equals
\begin{equation}\la{B3}
R(S,Y)=\fr 12(2\rho^{(1)}+\rho^{(2)},G\rho^{(2)}). 
\end{equation}

\subsection{The null space of the energy Hessian}

In this section we calculate the null space 
 \begin{equation}\la{KYd}
  \cK(S):=\Ker\, E''(S)\Big|_\cW,\qquad S\in \cS
  \end{equation}
 under the Wiener condition.

  \bl
  Let the Jellium and the Wiener conditions (\re{Wai}),  
(\re{W1}) hold and $S\in\cS$. Then
  \begin{equation}\la{KY}
  \cK(S)=\{(C,\ov s,0):~~ C\in\C,~~s\in\R^3 \},
  \end{equation}
  where $\ov s\in\R^{3\ov N}$ is defined similarly to (\re{gr2}): 
  $\ov s(n)\equiv s$.
  \el
  \begin{proof}
  All summands  of
the
energy (\re{B2}) are nonnegative. Hence,  this expression is zero if and only if
 all the summands vanish:
\begin{equation}\la{rb}
\vp(x)\equiv C,\quad
(\rho^{(1)}, G\rho^{(1)})=\Vert\sqrt{G}[\si^{(1)}-2e\psi_\al \cdot\vp]\Vert_{L^2({T_N})}^2
=0,\quad p=0.
\end{equation}
Note that $\sqrt{G}\psi_\al \cdot\vp=\sqrt{G}\psi_\al \cdot C=0$ since 
the operator $G$ annihilates the constant functions 
by (\re{fs}). Hence, (\re{rb}) implies that
\begin{equation}\la{rb2}
\sqrt{G}\si^{(1)}=0.
\end{equation}
On the other hand, (\re{ro11}) gives
in the Fourier transform 
\begin{equation}\la{B314}
\hat\si^{(1)}(\xi)=\hat\si(\xi)\xi\cdot\sum_{n\in{\Gamma_N}} ie^{i\xi (n+r)}\vka(n)
=i\hat\si(\xi)\xi\cdot e^{i\xi r}\hat \vka(\xi),\qquad\xi\in {\Gamma^*_N},
\end{equation}
where $\hat \vka(\xi):=\sum_{n\in{\Gamma_N}} ie^{i\xi n}\vka(n)$ is a $2\pi\Z^3$-periodic function on
${\Gamma^*_N}$.
Hence, definition (\re{fs}) and the Jellium condition (\re{Wai}) imply that
\begin{eqnarray}\la{B315}
0=\Vert\sqrt{G}\si^{(1)}\Vert_{L^2({T_N})}^2&=&
 N^{-3}\sum_{{\Gamma^*_N}\setminus {\Gamma^*_1}} |\hat\si(\xi)\fr{\xi\hat \vka(\xi)}{|\xi|}|^2
\nonumber\\
\nonumber\\
&=&N^{-3}\sum_{\theta\in\Pi^*_N\setminus {\Gamma^*_1}} 
\langle\hat \vka(\theta),
 \sum_{m\in\Z^3}\Big[\fr{\xi\otimes\xi}{|\xi|^2}|\hat\si(\xi)|^2\Big]_{\xi=\theta+2\pi m}\hat \vka(\theta)\rangle
 \nonumber\\
\nonumber\\
&=&N^{-3}\sum_{\theta\in\Pi^*_N\setminus {\Gamma^*_1}} 
\langle\hat \vka(\theta),
 \Si(\theta)
 \hat \vka(\theta)\rangle.
\end{eqnarray}
As a result, 
\begin{equation}\la{B316}
\hat \vka(\theta)=0,\qquad \theta\in \Pi^*_N\setminus {\Gamma^*_1}
\end{equation}
by the Wiener condition
(\re{W1}).   
On the other hand, $\hat \vka(0)\in\R^3$ remains arbitrary,
see Remark \re{rW} ii).
	Respectively, 
	$\vka=\ov s$ with 
	an arbitrary $s\in\R^3$.
\end{proof}
\br
{\rm 
The key point of the proof is 
the explicit calculation (\re{B314}) in the Fourier transform.
This calculation relies on the invariance of
the Hessian $E''(S)$ with respect to ${\Gamma_N}$-translations which is 
due to the periodicity of the ions arrangement of the ground state.
}
\er
\br ({\it Beyond the Wiener condition.})
{\rm
 If
the Wiener condition (\re{W1})  fails, the dimension of the space
\begin{equation}\la{V}
V:=\{v\in \R^{3\ov N}:~~ v(n)=\sum_{\theta\in\Pi^*_N\setminus{\Gamma^*_1}}
e^{-i\theta n} \hat v(\theta),
\qquad \hat v(\theta)\in\C^3,~~ \Si(\theta)\hat v(\theta) {\color{red} \equiv}\,  0\}
\end{equation}
is positive.
The above calculations show that in this case 
 \begin{equation}\la{KYg}
  \cK(S)=\{(C, \ov s+v,0):~~ C\in\C,~~s\in{T_N},~~v\in V \}.
  \end{equation}
 The subspace $V\subset \R^{3\ov N}$ is orthogonal to the $3D$ subspace 
$\{\ov s:s\in\R^3\}\subset \R^{3\ov N}$ by the Parseval theorem.
Hence,  $\dim \cK(S)=5+d$,  where $d:=\dim V>0$. 
Thus, $\dim \cK(S)>5$.
Under the Wiener condition $V=0$, and (\re{KYg}) coincides with (\re{KY}).
}
 \er

\subsection{The positivity  of the Hessian}
Denote by $N_S\cS$ the normal subspace to $\cS$ at a point $S$:
\begin{equation}\la{L0N}
N_S\cS:=\{Y\in\cW: \langle Y,\tau\rangle=0,~~\tau\in T_S\cS\},
\end{equation}
where $T_S\cS$ is the
tangent space to $\cS$ at the point {\color{red}$S$}  and $\langle\cdot,\cdot\rangle$
stands for the scalar product (\re{dWW}).

\bd\la{dM}
Denote by $\cM$ the Hilbert manifold 
\begin{equation}\la{cM}
\cM:=\{X\in\cV: Q(X)=Z N^3\}.
\end{equation}
\ed
Obviously, $\cS\subset\cM$, and a tangent space to $\cM$ at a point
$S=(\psi_\al,\ov r, 0)$
is given by 
 \begin{equation}\la{TSM}
T_S\cM=\{(\vp, \vka,\pi)\in\cW:\vp\bot\psi_\al,~~ \vka\in\R^{3\ov N},
~~\pi\in\R^{3\ov N} \},
\end{equation}
since $DQ(\psi_\al,\ov r,0)=(\psi_\al,0,0)$.
\bl
Let the Jellium condition (\re{Wai}) hold and $S=S_{\al ,r}\in \cS$. Then
the Wiener condition (\re{W1}) is necessary and sufficient for the positivity
of the Hessian $E''(S)$
in the orthogonal directions to $\cS$ on $\cM$, i.e., 
\begin{equation}\la{cM2}
E''(S)\Big|_{N_S\cS\cap T_S\cM}>0.
\end{equation}
\el
\begin{proof}
i) Sufficiency.
Differentiating $S_{\al,r}=(e^{i\al}\psi_0,\ov r,0)\in \cS$ in the parameters $\al \in [0,2\pi]$ and $r\in{T_N}$, 
we obtain 
\begin{equation}\la{tv}
T_S\cS=\{(iC\psi_\al ,\ov s,0): ~~C\in\R,~~s\in\R^3\}.
\end{equation}
Hence, (\re{KY})  implies that 
\begin{equation}\la{tv2}
K(S):=\cK(S)\cap N_S\cS=
\{(C\psi_\al ,0,0): ~~C\in\R\}
\end{equation}
by (\re{1i}) and (\re{dWW}).
Therefore,
	\begin{equation}\la{tv22}
	\cK(S)\cap  N_S\cS\cap  T_S\cM=K(S)\cap  T_S\cM=
	(0 ,0,0),
	\end{equation}
since the vector $(\psi_\al ,0,0)$
 is orthogonal to $T_S\cM$ by (\re{TSM}). Now (\re{cM2}) follows since 
$E''(S)\ge 0$ by (\re{B2}).
\medskip\\
ii)	Necessity.
If the Wiener condition (\re{W1}) fails,
the null space $\cK(S)$ is given by (\re{KYg}). Hence, 
(\re{tv}) implies that now
\begin{equation}\la{tv3}
K(S)=
\{(C\psi_\al ,v,0): ~~C\in\R, ~~ v\in V\}.
\end{equation}
However,  $(\psi_\al,\psi_\al)>0$.
Hence, (\re{TSM}) implies that $(\psi_\al ,v,0)\not\in T_S\cM$ 
and
the intersection
\begin{equation}\la{tv4}
K(S)\cap T_S \cM=
\{0,v,0): ~~ v\in V\}
\end{equation}
is the nontrivial subspace of the dimension  $d>0$. 
Thus, the Hessian $E''(S)$ vanishes on this nontrivial subspace
of  $N_S\cS\cap  T_S\cM$.
\end{proof}

\br\la{rS} {\rm The positivity of type (\re{cM2}) breaks down for the 
submanifold 
$\cS(r):=\{S_{\al ,r}:\al \in[0,2\pi]\}$ 
with a fixed $r\in{T_N}$ instead of  the solitary manifold $\cS$.
Indeed, in this case the corresponding tangent space  is smaller, 
\begin{equation}\la{tvr}
T_S\cS(r)=\{(iC\psi_\al ,0,0): ~~C\in\R\},
\end{equation}
and hence, the normal subspace $N_S\cS(r)$ is larger, containing all 
vectors $(0,\ov s,0)$
generating  the shifts of the torus. However, all these vectors also 
belong 
to the null space (\re{KY}) and to $T_S\cM$.
Respectively, the null space of the Hessian $E''(S)$ in $T_S\cM\cap N_S\cS(r)$ 
is  3-dimensional.
}
\er

\subsection{The orbital stability}
Here we prove our   main result.
\bt\la{tm}
Let the conditions (\re{Wai}),  (\re{W1}) and  (\re{ro+}) hold, 
and $\cS$ is the solitary 
manifold (\re{cS}). 
Then
for any 
$\ve>0$  there exists $\de=\de(\ve)>0$ such that for 
$X(0)\in\cM$ with
$d_\cV(X(0),\cS)<\de$ we have
\begin{equation}\la{m}
d_\cV(X(t),\cS)<\ve,\qquad t\in\R
\end{equation}
for the corresponding solution
$X(t)\in C(\R,\cV)$  to (\re{LPS1})\,-\,(\re{LPS3}).
\et
For the proof is suffices to check
the lower energy estimate (\re{BLi}): 
\begin{equation}\la{BL}
E(X)\ge \nu\,d^2_\cV(X,\cS)\quad{\rm if}\quad d_\cV(X,\cS)\le\de,\quad X\in\cM
\end{equation}
with some $\nu,\de>0$. 
This estimate implies Theorem \re{tm}, since the energy is conserved
along all trajectories.
First, we prove similar lower bound for the energy Hessian. 
\bl
Let all conditions of Theorem \re{tm}  hold. Then  for each $S\in \cS$
\begin{equation}\la{L02}
 \langle Y, E''(S)Y\rangle > \nu\Vert Y\Vert_\cW^2,  \qquad Y\in  N_S\cS \cap T_S\cM,
\end{equation}
where $\nu>0$.
\el
\begin{proof}
It suffices to prove (\ref{L02}) for $S=(\psi_0,0,0)$. Note that $E''(S)$ is not complex linear due to the integral 
in (\re{Hfor}). Hence, we express the action of $E''(S)$ in $\psi_1(x):=\rRe\psi(x)$ and $\psi_1(x):=\rIm\psi(x)$:
by the formula (1.15) of \ci{KKpl2015},
\begin{equation}\la{E''}
E''(S)Y=\left(\begin{array}{cccl}
 2H_0+4e^2\psi_0 G\psi_0 & 0 & 2L & 0
 \medskip\\
 0 & 2H_0  &0 & 0\medskip\\
 2L^{\5*}  &    0  &   T    & 0  \\
      0      &    0            &   0    &  M^{-1} \\
\end{array}\right)Y\qquad{\rm for}\quad
Y=\left(\begin{array}{c}\psi_1 \\ \psi_2 \\ q \\ p \end{array}\right),
\end{equation}
where $H_0:=-\fr12\De$ as in (\re{LPSi}), and $\psi_0$ denotes the operator of  multiplication
by the real function $\psi_0(x)\equiv\sqrt{Z}$. The operator $L$  corresponds to the matrix
\begin{equation}\la{S}
 L(x,n):=e\psi_0(x)G\na\si(x-n):~~ ~~x\in\R^3,~n\in\Gamma_N
\end{equation}
by formula (3.3) of \ci{KKpl2015} and  $T$ corresponds to the real matrix with the entries
\begin{equation}\la{T}
T(n-n'):=-\ds\langle  G\na\otimes\na\si(x-n'),  \si(x-n) \rangle,\quad n, n'\in\Gamma_N
\end{equation}
by formula (3.4) of \ci{KKpl2015} since the corresponding potential  $\Phi_0=0$.
Thus, $E''(S)$ is a finite-rank perturbation of the operator with the discrete spectrum on the torus ${T_N}$. 
Moreover,  (\re{cM2}) implies that the minimal eigenvalue of $E''(S)$ is  positive.
Therefore, (\re{L02}) follows.
\end{proof}
\medskip

The positivity (\re{L02}) implies
the lower energy estimate (\re{BL}),  since the 
higher-order terms in (\re{te}) 
are negligible by the following lemma.

\bl\la{lre}
Let $\si(x)$ satisfy  (\re{ro+}).
Then the  remainder (\re{B3}) admits the estimate
\begin{equation}\la{B31}
|R(S,Y)|\le C
\Vert Y\Vert_\cW^3\quad\,\,\,{\rm for}\quad  \,\,\,\Vert Y\Vert_\cW \le 1.
\end{equation}
\el
\begin{proof} 
It suffices to prove the 
estimates
\begin{equation}\la{B312}
\Vert\sqrt{G}\rho^{(1)}\Vert_{L^2({T_N})}\le C_1\Vert Y\Vert_\cW,\quad  
\Vert\sqrt{G}\rho^{(2)}\Vert_{L^2({T_N})}\le C_2\Vert Y\Vert_\cW^2
\quad   \,\,\,{\rm for}\,\,\, \quad    \Vert Y\Vert_\cW  \le 1.
\end{equation}
Then  (\re{B31}) will follow from (\re{B3}).
\medskip\\
i)
By (\re{ro11}) we have
for $Y=(\vp,\vka,p)$ 
\begin{equation}\la{B313}
\sqrt{G}\rho^{(1)}=\sqrt{G}\si^{(1)}-2e\sqrt{G}\psi_\al (x)\cdot\vp(x).
\end{equation}
The operator $\sqrt{G}$ is bounded in $L^2(\R^3)$ by the definition (\re{fs}). 
Hence,  
\begin{equation}\la{GP1}
\Vert\sqrt{G}\si^{(1)}\Vert_{L^2({T_N})}\le C |\vka|
\end{equation}
by (\re{ro11}). Applying to the second term the 
Cauchy-Schwarz and Hausdorff-Young inequalities,
we obtain 
\begin{equation}\la{GP2}
\Vert\sqrt{G}\psi_\al (x)\cdot\vp\Vert_{L^2({T_N})}\le 
C\Big[\sum_{\xi\in{\Gamma^*_N}}\fr{|\hat\vp(\xi)|^2}{|\xi|^2}\Big]^{1/2}\le 
C\Vert\hat\vp\Vert_{L^4({\Gamma^*_N})}\Big[\sum_{\xi\in{\Gamma^*_N}}|\xi|^{-4}\Big]^{1/2}
\le
C\Vert\vp\Vert_{L^{4/3}({T_N})}.
\end{equation}
Hence, the first inequality (\re{B312}) is proved.
\medskip\\
ii) Now we prove the second 
 inequality  (\re{B312}). By (\re{ro13}) we have
for $Y=(\vp,\vka,p)$ 
\begin{equation}\la{B317}
\sqrt{G}\rho^{(2)}(x)=\sqrt{G}\si^{(2)}(x)-e\sqrt{G}|\vp(x)|^2.
\end{equation}
Similarly to  (\re{GP1})
\begin{equation}\la{GP3}
\Vert\sqrt{G}\si^{(2)}\Vert_{L^2({T_N})}\le C|\vka|^2.
\end{equation}
 Finally, 
denoting $\beta(x):=|\vp(x)|^2$, we obtain similarly to (\re{GP2}) 
\begin{eqnarray}\la{fp}
\Vert\sqrt{G}|\vp(x)|^2\Vert_{L^2({T_N})}
&\le&
C\Big[\sum_{\xi\in{\Gamma^*_N}}\fr{|\hat \beta(\xi)|^2}{|\xi|^2}\Big]^{1/2}\le 
C\Vert \hat\beta\Vert_{L^4({\Gamma^*_N})}\Big[\sum_{\xi\in{\Gamma^*_N}}|\xi|^{-4}\Big]^{1/2}
\nonumber\\
&\le&
C_1\Vert \beta\Vert_{L^{4/3}({T_N})}
=
C_1\Vert \vp\Vert_{L^{8/3}({T_N})}^2
\le C_2\Vert\vp\Vert_{H^1({T_N})}^2
\end{eqnarray}
by the Sobolev embedding theorem {\color{red} \ci{Adams}}.
Now the lemma is proved.
\end{proof}

\appendix

\setcounter{section}{0}
\setcounter{equation}{0}
\protect\renewcommand{\thesection}{\Alph{section}}
\protect\renewcommand{\theequation}{\thesection.\arabic{equation}}
\protect\renewcommand{\thesubsection}{\thesection.\arabic{subsection}}
\protect\renewcommand{\thetheorem}{\Alph{section}.\arabic{theorem}}

\setcounter{equation}{0}
\section{Conservation laws}
We deduce the conservation  laws (\ref{EQ}) by the Galerkin approximations \ci{Lions}. 

\bd
i)$\cV_m$ with $m\in\N$  denotes finite dimensional  submanifold of $\cV$ formed by
\begin{equation}\la{Vm}
(\sum_{k\in{\Gamma^*_N}(m)} C_k e^{ikx},q,p),  \qquad q\in{T_N}^{\ov N}, \quad p\in\R^{3\ov N}.
\end{equation}
where ${\Gamma^*_N}(m):=\{k\in{\Gamma^*_N}:  k^2\le m\}$.
\medskip\\
ii) $\cW_m$ with $m\in\N$  denotes the finite dimensional linear subspace of $\cW$
spanned by 
\begin{equation}\la{Wm}
(\sum_{k\in{\Gamma^*_N}(m)} C_k e^{ikx},\vka,v),  \qquad \vka\in\R^{3\ov N}, \quad v\in\R^{3\ov N}.
\end{equation}
\ed
 Obviously, $\cV_1\subset\cV_2\subset...$, 
the union $\cup_m\cV_m$ is dense in $\cV$, and 
 $\cW_m$ are  invariant with respect 
to {\color{red}$A$} and $J$.
Let us denote by $P_m$
the orthogonal projector $\cX\to\cW_m$. This projector is also orthogonal in $\cW$.
Let us 
approximate the system 
(\re{HS}) by the finite dimensional Hamilton systems  on the manifold $\cV_m$,
\begin{equation}\la{gal}
\dot X_m(t)=JE_m'(X_m(t)),\qquad t\in\R,
\end{equation}
where $E_m:=E|_{\cV_m}$ and $X_m(t)=(\psi_m(t),q_m(t),p_m(t))\in C(\R,\cV_m)$.
The equation (\re{gal}) can be also  written as
\begin{equation}\la{gali}
\langle \dot X_m(t),Y\rangle=-\langle E'(X_m(t)),JY\rangle,\qquad Y\in\cW_m.
\end{equation}
This form of the equation (\re{gal}) holds since $E_m:=E|_{\cV_m}$ and $\cW_m$ is  invariant 
with respect to  $J$. Equivalently,
\begin{equation}\la{gali2}
\dot X_m(t)={\color{red}-A}\, X_m(t) + P_m N(X_m(t)).
\end{equation}

The Hamiltonian form guarantees the energy and charge conservation
 (\ref{EQ}):
\begin{equation}\la{EQ2}
E(X_m(t))=E(X_m(0)),\quad Q(X_m(t))=Q(X_m(0)),\qquad t\in\R.
\end{equation}
Indeed, the energy conservation holds by the Hamiltonian form  (\re{gal}), 
while the charge conservation holds
by
the Noether theory \ci{A,GSS87, KQ} due to the
$U(1)$-invariance of $E_m$, see (\re{U1}).

The equation (\re{gali2}) admits a unique local solution for every initial state
 $X_m(0)\in\cV_m$ since the right hand side 
is locally bounded and Lipschitz continuous. 
The global solutions exist by   (\re{EQV})
and
the energy and charge conservation (\re{EQ2}).
\medskip

Finally, we take any $X(0)\in\cV$ and choose a sequence 
\begin{equation}\la{s0}
X_m(0)\to X(0),\qquad m\to\infty,
\end{equation}
where the convergence holds in the metric of $\cV$.
 Therefore, 
\begin{equation}\la{EQm}
E(X_m(0))\to E(X(0)),\qquad Q(X_m(0))\to Q(X(0)).
\end{equation}
Hence, (\ref{EQ2}) and (\re{EQV}) imply the basic uniform bound 
\begin{equation}\la{Vb}
R:=\sup_{m\in\N}\,\,\sup_{t\in\R}| X_m(t)|_\cV <\infty.
\end{equation}
Therefore, 
(\ref{gali2}) and Lemma \re{p1} imply  the second basic uniform bound
\begin{equation}\la{Vb2}
\sup_{m\in\N}\,\sup_{t\in\R}\,\Vert \dot X_m(t)\Vert_{\cW^{-1}} <C(R),
\end{equation}
since  the operator 
${\color{red}A}:\cW\to\cW^{-1}$ is bounded, and the projector
$P_m$ is also a bounded operator in $\cW\subset \cW^{-1}$.
Hence, 
the Galerkin
approximations $X_m(t)$ are uniformly Lipschitz-continuous with values in $\cV^{-1}$:
\begin{equation}\la{ecg}
\sup_{m\in\N}\, d_{\cV^{-1}}(X_m(t),X_m(s))\le C(R)|t-s|,\qquad s,t\in\R.
\end{equation}
Let us show that  
the uniform estimates   (\ref{Vb}) and (\ref{ecg}) provide a compactness of the 
Galerkin approximations and the conservation laws. 
Let us recall that $\cX:=\cV^0$ and $\cV:=\cV^1$.
\bp\la{lgal} Let   (\re{ro+}) hold and $X(0)\in\cV$. Then 
\medskip\\
i) 
There exists
a subsequence $m'\to\infty$ such that
\begin{equation}\la{ss}
X_{m'}(t)\tocX X(t),\qquad m'\to\infty,\qquad t\in\R,
\end{equation}
where $X(\cdot)\in C(\R, \cX)$.
\medskip\\
ii) Every limit function $X(\cdot)$
is a solution to  (\re{LPSiv}), and $X(\cdot)\in C(\R,\cV)$.
\medskip\\
iii) The conservation laws (\re{EQ}) hold.
\ep
\begin{proof}
i) The convergence (\re{ss}) follows from 
 (\re{Vb}) and  (\re{Vb2}) 
by the Dubinsky
 `theorem on three spaces' \ci{Dub65}  (Theorem 5.1 of
\ci{Lions}). Namely, the embedding $\cV\subset\cX$ is compact by the Sobolev theorem 
{\color{red} \ci{Adams}},
and hence, (\re{ss}) holds by 
(\re{Vb})
for $t\in D$, where $D$ is a countable dense set. 
Finally, let us use the  interpolation inequality  and  (\re{Vb}), (\re{ecg}):
for any $\ve>0$
\begin{equation}\la{inti}
 d_\cX(X_m(t),X_m(s))\le\ve d_\cV(X_m(t),X_m(s))
+ C(\ve)  d_{\cV^{-1}}(X_m(t),X_m(s))\le 2\ve R+C(\ve,R)|t-s|.
\end{equation}
This inequality 
implies the 
equicontinuity of the Galerkin approximations with the values in $\cX$. Hence,
convergence 
(\re{ss}) holds for all $t\in\R$ since it holds for 
the dense set of $t\in D$. 
The same equicontinuity also implies  the continuity of the limit function $X\in C(\R, \cX)$.
\medskip\\
ii) 
Integrating  equation (\re{gali2}), we obtain 
\begin{equation}\la{gal2}
\int_0^t\langle \dot X_m(t),Y\rangle\,ds={\color{red}-}\int_0^t
\langle X_m(s),{\color{red} A}Y)\,ds + \int_0^t\langle N(X_m(s)),Y\rangle\,ds,
\qquad Y\in\cW_m,
\end{equation}
Below we will write $m$ instead of $m'$. 
To prove   (\re{LPSiv}) it suffices to check that  in the limit $m\to\infty$, we get
\begin{equation}\la{gal4}
\int_0^t\langle \dot X(t),Y\rangle\,ds
={\color{red}-}\int_0^t\langle X(s),{\color{red} A}Y\rangle\,ds + \int_0^t\langle N(X(s)),Y\rangle\,ds,
\qquad Y\in\cW_n,\qquad n\in\N.
\end{equation}
The convergence of  the left hand side and of the first term on the right hand side 
of (\re{gal2}) follow from (\re{ss}) and  (\re{s0}) since ${\color{red} A}Y\in\cW_m$.

It remains to consider the last integral of (\re{gal2}).
The integrand is uniformly bounded by (\re{Vb}) and Lemma \re{p1}. 
Hence, it suffices to check the pointwise convergence
\begin{equation}\la{Nm}
\langle N(X_m(t), Y\rangle-\!\!\!-\!\!\!\!\to \langle N(X(t), Y\rangle,\quad m\to\infty,\qquad Y\in\cW_n
\end{equation}
 for any $t\in\R$. Here $N(X_m(t))=(ie\Phi_m(t)\psi_m(t),p_m(t),f_m(t))$ according to 
the notations (\re{HN}), and $Y=(\vp,\vka,v)\in \cW_n$. Hence, (\re{Nm}) reads
 \begin{equation}\la{Nm2}
 ie(\Phi_m(t)\psi_m(t),\vp)+p_m(t)\vka+f_m(t)v\,\to\,  ie(\Phi(t)\psi(t),\vp)+p(t)\vka+f(t)v,\quad m\to\infty.
\end{equation}
 The convergence of $p_m(s)\vka$ follows from (\re{ss}) (with $m'=m$) .
 To prove the convergence of two remaining terms, we first show that
\begin{equation}\la{2rt}
\Phi_m(t):=G\rho_m \toLC \Phi(t):=G\rho,\quad m\to\infty.
\end{equation}
Indeed,  (\re{ss}) implies that 
\begin{equation}\la{qp}
\psi_m(t)\toLt\psi(t),\qquad q_m(t)\to q(t),\quad m\to\infty.
\end{equation}
	The sequence $\psi_m(t)$ is bounded in $H^1({T_N})$ by (\re{Vb}).
	Hence, $\psi(t)\in H^1({T_N})$ and 
	the sequence $\rho_m(t)$ is bounded in the Sobolev space $W^{1,3/2}({T_N})$ by (\re{np}).
	Therefore, the sequence $\rho_m(t)$ is precompact in $L^2({T_N})$ by  the Sobolev compactness theorem. 
	Hence,
	\begin{equation}\la{qp2}
	\rho_m\toLt\rho,\quad m\to\infty
	\end{equation}
	by (\re{qp}).
	Therefore, (\re{2rt}) holds since the operator $G:L^2({T_N})\to C({T_N})$ is continuous. 
	From (\re{2rt}) and (\re{qp}) it follows  that
	\begin{equation}\la{Nm3}
	\Phi_m(t)\psi_m(t)\toLt\5 \Phi(t)\psi(t),\quad f_m(t)\to f(t),\quad m\to\infty,
	\end{equation}
  which proves (\re{Nm2}). Now (\re{gal4}) is proved for $Y\in\cV_n$ with any $n\in\N$.
 Hence, $X(t)$ is a solution to (\re{HS}).  Finally,  
	$\Vert N(X(\cdot))\Vert_\cW$ is a bounded function  by
  (\re{Vb}) and Lemma \re{p1}.
  Hence, 
 (\re{LPSiv}) implies that  $X(\cdot)\in C(\R,\cV)$.
 \medskip\\
 iii) The conservation laws (\re{EQ2}) and the convergences (\re{s0}), (\re{ss}) imply that
 \begin{equation}\la{EQ3}
E(X(t))\le E(X(0)),\quad Q(X(t))\le Q(X(0)),\qquad t\in\R.
\end{equation}
The last inequality holds by the first convergence of (\re{qp}). The first inequality follows from the representation
\begin{equation}\la{EQ4}
E(X_m(t))=\fr12 \Vert\na\psi_m(t)\Vert_{L^2({T_N})}^2+\fr12 
\Vert \sqrt{G}\rho_m(t)\Vert_{L^2({T_N})}^2 +\sum_{n\in\Gamma_n}\fr{p_m^2(n,t)}{2M}.
\end{equation}
Namely, the last two terms on the right hand side converge by (\re{qp2}) and (\re{ss}). Moreover, 
the first term is bounded by (\re{Vb}). Hence,  the first convergence of  (\re{qp}) 
implies the weak convergence 
 \begin{equation}\la{EQ5}
\na\psi_{m}(t)\toLwt\na\psi(t)
 \end{equation}
 by the Banach theorem.
Now the first inequality of (\re{EQ3}) follows by the property of the weak convergence in the Hilbert space.
Finally, the opposite inequalities to (\re{EQ3}) are also true by the uniqueness of 
solutions $X(\cdot)\in C(\R,\cV)$, which is proved in Proposition \re{TLWP}.
\end{proof}

\setcounter{equation}{0}
\section{Jellium ground states}
We describe all 
 solutions to (\re{LPS1})\,-\,(\re{LPS3}) with minimal energy  (\re{Hfor}),
give some examples of ion densities illustrating the Jellium and 
the Wiener conditions, 
and show the existence of non-periodic ground states.

\subsection{Description of all ground states}

The following lemma gives the description of all ground states.  
\bl\la{Jgs}
Let the  Jellium condition (\re{Wai}) hold. Then 
all solutions 
to (\re{LPS1})\,-\,(\re{LPS3})
of minimal (zero) energy are $(\psi_\al,q^*,0)$ 
with $q^*\in{T_N}^{\ov N}$ satisfying 
the identity
\begin{equation}\la{sipiq}
	\sum_{n\in{\Gamma_N}}\si(x-q^*(n))\equiv eZ,\qquad x\in{T_N},
\end{equation}

\el
\begin{proof}
First,
let us note that the ${\Gamma_N}$-periodic solutions (\re{gr}) 
have the zero energy, and 
the identity (\re{sipiq}) holds for $q^*=\ov r$ by  (\re{sipi}).

Further, for any solution with zero energy  (\re{Hfor})
all summands on the right hand side of  (\re{Hfor}) vanish.
The first integral  vanishes only
for constant functions. Hence, the normalization condition (\re{rQ}) gives
  \begin{equation}\la{ppo}
  \psi(x,t)\equiv
\psi_{\al(t)}(x)\equiv e^{i\al(t)}\sqrt{Z},\qquad\al(t)\in \R.
\end{equation}
Then 
\begin{equation}\label{roZ2}
 \rho^e(x,t):=-e|\psi_{\al(t)}(x)|^2\equiv -eZ,\qquad x\in {T_N},\,\,\,\,t\in\R,
 \end{equation}
similarly to (\re{roZ}).
Further,
the second summand of (\re{Hfor}) vanishes only for $\rho(x,t)\equiv 0$
that is equivalent to (\ref{sipiq}) 
with $q(n,t)$ instead of $q^*(n)$
by (\re{roZ2}).
However, $\pa_t q(n,t)=p(n,t)\equiv 0$  for the zero energy  (\re{Hfor}).
Hence,
\begin{equation}\la{qnc}
q(n,t)\equiv q^*(n),\qquad t\in\R,
\end{equation}
where $q^*$ satisfies  (\re{sipiq}).
Moreover, $\Phi(x,t)\equiv 0$ 
 by the Poisson equation (\re{LPS2}) with $\rho(x,t)\equiv 0$.
Hence,   finally,
substituting  (\re{ppo}) into  (\re{LPS1}) 
with $\Phi(x,t)\equiv 0$,
we obtain that $\al(t)\equiv{\rm const}$.
 \end{proof}

This lemma implies that all  ${\Gamma_N}$-periodic ground states are given by
(\re{gr}).

 
 \subsection{Jellium and Wiener conditions. Examples}\la{sex}

The Wiener condition (\re{W1}) for the ground states (\re{gr}) holds 
under the generic assumption
(\re{W1s}).
On the other hand,  (\re{W1}) does not hold
  for the simplest Jellium model,
 when $\si(x)$ is the characteristic function
\begin{equation}\la{sic}
\si(x)=\si_1(x):=\left\{
\begin{array}{ll}
eZ,& x\in\Pi\\ 
0,& x\in{T_N}\setminus\Pi
\end{array}\right|{\color{red},}
\end{equation}
{\color{red} where $\Pi:=[-1/2,1/2]^3$}.
 Indeed, 
in this case the Fourier transform 
\begin{equation}\la{sJM}
\hat\sigma_1(\xi)=eZ\hat\chi_1(\xi_1)\hat\chi_1(\xi_2)\hat\chi_1(\xi_3);\qquad
\hat\chi_1(s)=\fr {2\sin s/2}s,\quad s\in\R\setminus 0,
\end{equation}
where 
$\chi_1(s)$ is the characteristic function of the interval 
{\color{red} $[-1/2,1/2]$}.
In this case we have
for $\theta= (0,\theta_2,\theta_3)$,
\begin{equation}\la{DK2}
\Si(\theta)= \sum_{m\in\Z^3:\,m_1=0}\Big[
 \fr{\xi\otimes\xi}{|\xi|^2}|\hat\si(\xi)|^2\Big]_{\xi=\theta+2\pi m},
\end{equation}
which is a degenerate matrix since $\xi_1=0$ in each summand. Hence, (\re{W1}) fails.
Similarly, the Wiener condition  fails for 
$\si_k(x)=eZ\chi_k(x_1)\chi_k(x_2)\chi_k(x_3)$ where 
$\chi_k={\color{red}\chi_1*...*\chi_1}$ ($k$ times)
 with $k=2,3,...$,
since in this case
\begin{equation}\la{sJM2}
\hat\sigma_k(\xi)=eZ\hat\chi_k(\xi_1)\hat\chi_k(\xi_2)\hat\chi_k(\xi_3);\qquad
\hat\chi_k(s)=\Big[\fr {2\sin s/2}s\Big]^k,\quad s\in\R\setminus 0.
\end{equation}

 
 \subsection{Non-periodic ground states}\la{snp}

It is easy to construct ground states   which are not $\Gamma_1$-periodic
in the case  of characteristic function (\re{sic}). Namely, the identity (\re{sipiq}) obviously holds 
for periodic arrangement  of ions (\re{gr2}). Now let us modify this periodic arrangement as follows:
\begin{equation}\la{mod}
q^*(n)=
(r_1,r_2,r_3+\tau(n_1,n_2)),\qquad n\in{\Gamma_N},
\end{equation}
where $\tau(n_1,n_2)$ is an arbitrary point of
the circle
$\R/N\R$. Now (\re{sipiq}) obviously holds for any arrangement of ions (\re{mod}).

Next lemma gives a more general spectral 
assumptions on $\si$ which provide ground states
with non-periodic ion arrangements. For example, let us assume that
\begin{equation}\la{spc}
\si(\xi)=0,\quad \xi_3\in 2\pi\Z\setminus 0, \qquad{\rm and}
\qquad\si(\xi_1,\xi_2,0)=0,\quad (\xi_1,\xi_2)\in 2\pi\Z^2\setminus 0.
\end{equation}
In particular, this holds for the densities (\re{sJM2}) with all $k=1,...$

\bl\la{lnonp}
Let $\si$ satisfy the spectral condition (\re{spc}).
Then there exist  ground states  which are not $\Gamma_1$-periodic.

\el
\begin{proof}
In the Fourier transform (\re{sipiq}) reads
\begin{equation}\la{sip4}
\int_{T_N} e^{i\xi x}\sum_{n\in{\Gamma_N}}\si(x-n-q^*(n))dx=
\hat\si(\xi)\sum_{n\in{\Gamma_N}} e^{i\xi(n+q(n))}=
\left\{
\begin{array}{rl}
eZ\ov N,&\xi=0\\
0,&\xi\in{\Gamma^*_N}\setminus 0.
\end{array}\right.
\end{equation}
These identities hold for any density $\si$ satisfying (\re{Wai})
if 
\begin{equation}\la{sip5}
\sum_{n\in{\Gamma_N}} e^{i\xi(n+q^*(n))}=0,\qquad \xi\in{\Gamma^*_N}\setminus{\Gamma^*_1}.
\end{equation}
In particular, $q^*=\ov r$
satisfies the system
(\re{sip5})  since then
\begin{equation}\la{sip6}
\sum_{n\in{\Gamma_N}} e^{i\xi(n+q^*(n))}=\sum_{n\in{\Gamma_N}} e^{i\xi(n+r)}=
e^{i\xi r}
\sum_{n\in{\Gamma_N}} e^{i\xi n}=0,\qquad \xi\in{\Gamma^*_N}\setminus{\Gamma^*_1}.
\end{equation}
Indeed, 
\begin{equation}\la{sip7}
\sum_{n\in{\Gamma_N}} e^{i\xi n}=\sum_{n\in{\Gamma_N}} e^{i(\xi_1 n_1+\xi_2 n_2+\xi_3 n_3)}
=\sum_{n_1=0}^{ N-1} e^{i\xi_1 n_1}
\sum_{n_2=0}^{ N-1} e^{i\xi_2 n_2}
\sum_{n_3=0}^{ N-1} e^{i\xi_3 n_3}
=0
\end{equation}
since at least one $\xi_k\not\in 2\pi\Z$ for $\xi\in{\Gamma^*_N}\setminus{\Gamma^*_1}$.
Now we modify $\ov r$ as follows:
\begin{equation}\la{modr}
q^*(n):=( a_1(n_1,n_2)),a_2(n_1,n_2), r_3),\qquad n\in{\Gamma_N},
\end{equation}
where $a_1(n_1,n_2))$ and $a_2(n_1,n_2)$ are {\it arbitrary points} of the circle $R/2\pi\Z$.
Then for $\xi\in{\Gamma^*_N}\setminus{\Gamma^*_1}$ we have
\begin{equation}\la{sip62}
\sum_{n\in{\Gamma_N}} e^{i\xi(n+q^*(n))}=\sum_{n_1,n_2=0}^{N-1}
e^{i(\xi_1(n_1+a_1(n_1,n_2)+\xi_2(n_2+a_2(n_1,n_2)}
\sum_{n_3=0}^{N-1} e^{i\xi_3 n_3}=0
\quad{\rm if}\quad \xi_3\not\in  2\pi\Z.
\end{equation}
Hence,
all identities (\re{sip4}) hold by (\re{spc}).
\end{proof}

\subsection{On the problem of periodicity}\la{spp}
In the non-periodic examples above 
the Wiener condition fails. We suppose that the Wiener condition provides
only periodic ground states, however this is an open problem. 
For densities $\si$ satisfying a more strong condition (\re{W1}), the  
identity (\re{sipiq}) is equivalent  to the system (\re{sip5}) by (\re{sip4}).
The system 
 (\re{sip5})
can be written as an `algebraic system'
\begin{equation}
\sum_{n\in{\Gamma_N}} w_1^{m_1}(n) w_2^{m_2}(n) w_3^{m_3}(n)=0,\qquad m\in\Z^3\setminus NZ^3.
\end{equation}
for 
\begin{equation}
w_j(n_1,n_2,n_3):=e^{i\ds\fr{2\pi}N [n_j+q^*_j(n)]},
\qquad n\in{\Gamma_N},\quad j=1,2,3.
\end{equation}
The ${\Gamma_N}$-periodicity of $q^*$ 
is equivalent to the fact that
 only solutions are 
\begin{equation}
w_j(n_1,n_2,n_3)=C_j \lam^{n_j},\qquad j=1,2,3,
\end{equation}
where 
$\lam:=e^{i\ds\fr{2\pi}N}$. For the corresponding 1D analog
\begin{equation}
\sum_{n_1=0}^{N-1} w_1^{m_1}(n_1)=0,\qquad m=1,...,N-1,
\end{equation}
the only solutions are $w_1(n_1):=C_1 \lam^{n_{\color{red} 1}}$ that follows easily 
from the 
Newton-Girard formulas guessed by Girard in 1629 
and rediscovered (without a proof) by Newton in 1666. The formulas were 
proved by Euler in 1747, see \ci{Euler1747}.


\end{document}